\documentclass[12pt,draftcls,onecolumn]{IEEEtran}
\usepackage{graphicx,amsmath,amssymb,amsthm}
 \usepackage{color}
\newtheorem{prop}{Proposition}
\newtheorem{thm}{Theorem}

\newtheorem{rem}{Remark}
\newtheorem{example}{Example}
\newtheorem{defi}{Definition}
\usepackage{amsmath}
\usepackage{amssymb}

\newcommand{\kw}{\rule{2mm}{2mm}}
\newcommand{\ged}{\hfill\kw}

\begin{document}
\title{Risk-Sensitive Mean Field Games}
\author{Hamidou Tembine, Quanyan Zhu, Tamer Ba\c{s}ar
\thanks{We are grateful to many seminar and conference participants such as those in the Workshop on Mean Field Games (Rome, Italy, May 2011) and IFAC World Congress (Milan, Italy, August-September 2011) for their valuable comments and suggestions on the preliminary versions of this work. }
\thanks{ An earlier version of this work appeared in the Proceedings of 18th IFAC World Congress (Milan, Italy; August 29 - September 2, 2011).
Research of second and third authors was supported  in part by the U.S. Air Force Office of Scientific Research (AFOSR) under the MURI Grant FA9550-10-1-0573. The first author acknowledges the financial support from the CNRS mean-field game project ``MEAN-MACCS".
}

\thanks{H. Tembine is with Ecole Sup\'erieure d'Electricit\'e (SUPELEC), France.
        {\tt\small E-mail: tembine@ieee.org}}
        \thanks{
        Q. Zhu and T. Ba\c{s}ar are with Coordinated  Science Laboratory and the Department of Electrical and Computer Engineering, University of Illinois at Urbana-Champaign, Urbana, IL, USA.  {\tt\small \{zhu31, basar1\}@illinois.edu}}
}

\maketitle

\begin{abstract}
In this paper, we study a class of  risk-sensitive mean-field stochastic differential games.
We  show that under appropriate regularity conditions, the mean-field value of the stochastic differential game with  exponentiated integral cost functional coincides with the value function described by a Hamilton-Jacobi-Bellman (HJB)  equation with an additional quadratic term. We provide an explicit solution of the mean-field best response when the instantaneous cost functions are log-quadratic and the state dynamics are affine in the control. An equivalent mean-field risk-neutral problem is formulated and the corresponding mean-field equilibria are characterized in terms of backward-forward macroscopic McKean-Vlasov equations, Fokker-Planck-Kolmogorov equations, and HJB equations. We provide numerical examples on the mean field behavior to illustrate  both linear and McKean-Vlasov dynamics.

\end{abstract}

\section{Introduction}

Most formulations of mean-field (MF) models such as  anonymous sequential population games \cite{jovanovic,bergin}, MF stochastic controls \cite{m2,mrp2,meyn},  MF optimization, MF teams~\cite{tem1},  MF stochastic games \cite{weintraub,johari,tem1,noisy}, MF  stochastic difference games \cite{huang2011}, and
MF stochastic differential games \cite{lasry1,geant,mfpower} have been of risk-neutral type where the cost (or payoff, utility) functions to be minimized (or to be maximized) are the expected values of  stage-additive loss functions.

Not all behavior, however, can be captured by risk-neutral cost functions. One way of capturing risk-seeking or risk-averse behavior is by exponentiating loss functions before expectation (see~\cite{Basar99,jacobson} and the references therein).

 The particular risk-sensitive mean-field stochastic differential game that we consider in this paper
involves an exponential term in the stochastic long-term cost function. This approach was first taken
by  Jacobson in \cite{jacobson}, when considering the risk-sensitive Linear-Quadratic-Gaussian (LQG) problem with
state feedback. Jacobson demonstrated a link between the exponential cost
criterion and deterministic linear-quadratic differential games. He showed that the
risk-sensitive approach provides a method for varying the  robustness
of the controller and noted that in the case of no risk, or risk-neutral case, the
well known LQR solution  would result  (see, for follow-up work on risk-sensitive stochastic control problems with noisy state measurements, \cite{whittle,bensoussan,panbasar}).

%

In this paper, we examine the risk-sensitive stochastic differential game in a regime of large population of players. We first present a mean-field stochastic differential game model where the players are coupled not only via their risk-sensitive cost functionals but also via their states. The main coupling term is the mean-field process, also called the \textit{occupancy process} or \textit{population profile process}. Each player  reacts to the mean field or a subset of the mean field generated by the states of the other players in an area, and at the same time the mean field evolves according to a controlled Kolmogorov forward equation.

Our contribution can be summarized as follows.
Using a particular structure of state dynamics, we derive the mean-field limit of the individual state dynamics leading to a non-linear controlled macroscopic McKean-Vlasov equation; see \cite{kurtz}. Combining this  with a limiting risk-sensitive cost functional, we arrive at the mean-field response framework, and establish its compatibility with the density distribution using the controlled Fokker-Planck-Kolmogorov forward equation. The mean-field equilibria are characterized by coupled backward-forward equations. In general a backward-forward system may not have  solution (a simple example is provided in section~\ref{counter2}).
 An explicit solution of the Hamilton-Jacobi-Bellman (HJB) equation is provided for the affine-exponentiated-Gaussian mean-field  problem. An equivalent risk-neutral mean-field problem (in terms of value function) is formulated and the solution of the mean-field game problem is characterized. Finally, we provide a sufficiency condition for having at most one smooth solution to the risk-sensitive mean field system in the local sense.

The rest of the paper is organized as follows. In  Section \ref{tsec1}, we present the model description. We provide an overview of the mean-field convergence result in Section \ref{tsec2}. In Section \ref{tsec3}, we present the risk-sensitive mean-field stochastic differential game formulation and its equivalences.
In Section \ref{uni}, we analyze a special class of risk-sensitive  mean-field games where the state dynamics are linear and independent of the mean field. In Section \ref{numerical}, we provide a numerical example, and
section \ref{tsec5} concludes the paper.
An appendix includes proofs of two main results in the main body of the paper.
We summarize some of the notations used in the paper in Table \ref{tablenot2}.

\begin{table}[htb]
\caption{Summary of Notations} \label{tablenot2}
\begin{center}
{\color{black}
\begin{tabular}{ll}
\hline
  Symbol & Meaning \\ \hline
  $f$ & drift function (finite dimensional) \\
  $\sigma$ & diffusion function (finite dimensional)\\
  ${x}^n_j(t)$ &  state of Player $j$ in a population of size $n$\\
  $\bar{x}_j(t)$ & solution of macroscopic McKean-Vlasov equation\\
  ${x}_j(t)$ & limit of state process ${x}^n_j(t)$ \\
  $U_j$ & space of feasible control actions of Player $j$\\
$\tilde{\gamma}_j$ & state feedback strategy of Player $j$\\
$\bar{\gamma}_j$ & individual state-feedback strategy of Player $j$\\
$\tilde{\Gamma}_j$ & set of admissible state feedback strategies of Player $j$\\
$\bar{\Gamma}_j$ & set of admissible individual state-feedback strategies of Player $j$\\
$u_j$ & control action of Player $j$ under a generic control strategy\\
  $c$ & instantaneous cost function\\
  $g$ & terminal cost function\\
    $\delta$ & risk-sensitivity index\\
    $\mathbb{B}_j$ & standard Brownian motion process for Player $j$'s dynamics\\
  $\mathbb{E}$ & Expectation operator\\
  $L$ & risk-sensitive cost functional\\
  $\partial_x$ & partial derivative with respect to $x$ (gradient)\\
  $\partial^2_{xx}$ & second partial derivative (Hessian operator) with the respect to $x$\\
  $x'$ &  transpose of $x$ \\
    $m^n_t$ &  empirical measure of the states of the players \\
        $m_t$ &  limit of $m_t^n$ when $n\rightarrow\infty$ \\
                $m^n$ &  limit of $m_t^n$ when $t\rightarrow\infty$ \\
  tr($M$) & trace of a square matrix $M$, i.e., $ tr(M):= \sum_{i} M_{ii}.$\\
    $A \succ B$ & $A- B$ is positive definite, where $A$, $B$ are square symmetric matrices of the same dimension.
  \\ \hline
\end{tabular}
}
\end{center}
\end{table}

\section{The problem setting} \label{tsec1}
We consider a class of $n-$person stochastic differential games, where Player $j$'s individual state, $x_j^n$,  evolves according to the It\^o stochastic differential equation (S) as follows:
$$\begin{array}{ccl}
\nonumber dx_j^n(t)&=&\frac{1}{n}\displaystyle\sum_{i=1}^n f_{ji}(t, x_j^n(t),u^n_j(t),x_i^n(t)) dt + \frac{\sqrt{\epsilon}}{n}\displaystyle\sum_{i=1}^n \sigma_{ji}(t, x_j^n(t),u^n_j(t),x_i^n(t)) d\mathbb{B}_{j}(t),\\
\nonumber x_j^n(0)&=&x_{j,0}\in \mathcal{X}\subseteq\mathbb{R}^{k},\ k\geq 1,
j\in\{1,\ldots,n\},
\end{array} \ (\textrm{S})$$
where $x_j^n(t)$ is the $k$-dimensional state of Player $j$;
$u_j^n(t)\in {U}_j,$ is the control of Player $j$ at time $t$ with ${U}_j$ being a subset of the $p_j$-dimensional Euclidean space $\mathbb{R}^{p_j}$; $\mathbb{B}_j(t)$ are mutually independent standard Brownian motion processes in $\mathbb{R}^k$; and $\epsilon$ is a small positive parameter, which will play a role in the analysis in the later sections. We will assume in (S) that there is some symmetry in $f_{ji}$ and $\sigma_{ji}$, in the sense that there exist $f$ and $\sigma$ (conditions on which will be specified shortly) such that for all $j$ and $i$, 
$$f_{ji}(t, x_j^n(t),u^n_j(t),x_i^n(t)) \equiv f(t, x_j^n(t),u^n_j(t),x_i^n(t))$$
and
 $$\sigma_{ji}(t, x_j^n(t),u^n_j(t),x_i^n(t)) \equiv \sigma(t, x_j^n(t),u^n_j(t),x_i^n(t))\,.$$
The system (S) is a controlled McKean-Vlasov dynamics.
 {\color{black}
 Historically, the McKean-Vlasov stochastic differential equation (SDE) is a kind of mean field forward SDE suggested by Kac in 1956 as a stochastic
toy model for the Vlasov kinetic equation of plasma and the study of which was initiated by McKean in 1966. Since
then, many authors have made contributions to McKean-Vlasov type SDEs and related applications \cite{kac1,kac2}.}


{\color{black}
The uncontrolled version of state dynamics (S) captures many interesting problems involving interactions between agents. We list below a few examples.

\begin{example}[Stochastic Kuramoto model]
Consider $n$ oscillators where each of the oscillators is considered to have its own intrinsic natural frequency $\omega_j$, and each is coupled symmetrically to all other oscillators.
For $f_{ji}(x_i,u_i,x_j)= f(x_i,u_i,x_j)= K\sin(x_j-x_i) + \omega_j$ and $\sigma_{ji}$  a constant in (S), the state dynamics without control is known as (stochastic) Kuramoto oscillator \cite{kuramoto} where the  goal  is  convergence to some common value (consensus) or  alignment of the players' parameters. The stochastic Kuramoto model is given by
    $${d\theta_j}(t) = \left( \omega_{j}(t) +\dfrac{K}{n}\sum_{i=1}^n\sin(\theta_{i}(t)-\theta_{j}(t))\right) dt +Dd\mathbb{B}_{j}(t),
$$
where $D,K>0.$

\end{example}
\begin{example}[Stochastic Cucker-Smale dynamics:] Consider a population, say of birds or fish that  move in the three dimensional space. It has been observed that for some initial conditions,
for example on their positions and velocities, the state of the flock converges to one
in which all birds fly with the same velocity. See, for example, Cucker-Smale flocking dynamics \cite{cs07,cs07v2} where each vector $x_i=(y_i,v_i)$ is composed of position dynamics and velocity dynamics of the corresponding player.
    For $f(x_i,u_i,x_j)=(\epsilon^2+\parallel x_j-x_i\parallel^2)^{-\alpha}c(x_j-x_i)$ in (S), where $\epsilon>0,\alpha>0$ and $c(\cdot)$ is a continuous function, one arrives at a generic class of consensus algorithms developed for flocking problems.
\end{example}

\begin{example}[Temperature dynamics for energy-efficient buildings]
Consider a heating system serving a finite number of zones. In each zone, the goal is to maintain a certain temperature. Denote by $T_j$ the temperature of zone $j,$ and by $T^{ext}$ the ambient temperature.
The law of conservation of energy can be written down as the following  equation for zone $j,$
$$dT_{j}(t)=\sigma d\mathbb{B}_{j}(t)+ \left[r_{j}(t)+\frac{\gamma}{\beta}(T^{ext}(t)-T_{j}(t))+\sum_{i\neq j}\alpha_{ij}(t)(T_{i}(t)-T_{j}(t))\right] dt,
$$ where $r_{j}$ denotes the heat
input rate  of the heater in  zone $j,$ $\gamma,\beta>0,$ $\alpha_{ij}$ is the  thermal conductance between
zone $i$ and zone $j$ and $\sigma$ is a small variance term. The evolution of the temperature has a McKean-Vlasov structure of the type in system (S). We can introduce a control variable into $r_{j}$ such that the heater can be turned on and off in each  zone.
\end{example}

The three examples above can be viewed as special cases of the system (S). 
The controlled dynamics in (S) allows one to address several interesting questions. For example, how to control the flocking dynamics and consensus algorithms of the first two examples above to a certain target? How to control the temperature in the third example in order to achieve a specific thermal comfort while minimizing energy cost?
In order to define the controlled dynamical system in precise terms, we  have to specify the nature of information that players are allowed in the choice of their
control at each point in time. This brings us to the first definition below.}

\begin{defi}
A {\it state-feedback strategy} for Player $j$ is a mapping $\tilde{\gamma}_j:\ \mathbb{R}_{+}\times (\mathbb{R}^k)^n\longrightarrow {U}_j$, whereas
an {\it individual state-feedback strategy} for Player $j$ is a mapping $\bar{\gamma}_j:\ \mathbb{R}_{+}\times \mathbb{R}^k\longrightarrow {U}_j.$
\end{defi}
Note that the individual state-feedback strategy involves only the self state of a player, whereas the state-feedback strategy involves the entire $nk-$dimensional state vector. The individual strategy spaces in each case have to be chosen in such a way that the resulting system of stochastic differential equations (S) admits a unique solution (in the sense specified shortly) when the players pick their strategies independently; furthermore, the feasible sets are time invariant and independent of the controls. We denote by $\bar{\Gamma}_j$ the set of such admissible control laws $\bar{\gamma}_j: [0, T]\times \mathbb{R}^k \rightarrow {U}_j$ for Player $j$; a similar set, $\tilde{\Gamma}_j$, can be defined for state-feedback strategies $\tilde{\gamma}_j$.

We assume the following standard conditions on $f,  \sigma, \bar{\gamma}_j$ and the action sets $U_j$, for all $j=1, 2, \cdots, n$.

\begin{itemize}
\item[(i)] $f$ is $C^1$ in $(t,x,u,m)$, and Lipschitz in $(x,u,m)$.
\item[(ii)] The entries of the matrix $\sigma$ are $C^2$ and $\sigma \sigma'$ is strictly positive;
    \item[(iii)] $f,\partial_xf$ are uniformly bounded;
        \item[(iv)] ${U}_j$ is non-empty, closed and bounded;
        \item[(v)] $\bar{\gamma}_j:\ [0,T]\times \mathbb{R}^k\longrightarrow  {U}_j$ is piecewise continuous in $t$ and Lipschitz in $x.$
\end{itemize}

Normally, when we have a cost function for Player $j$, which depends also on the state variables of the other players, either directly, or implicitly through the coupling of the state dynamics (as in (S)), then any state-feedback Nash equilibrium solution will generally depend not only on self states but also on the other states, i.e., it will not be in the set $\bar{\Gamma}_j, j=1, \cdots, n$. However, this paper aims to characterize the solution in the high-population regime (i.e., as $n\rightarrow\infty$) in which case the dependence on other players' states will be through the distribution of the player states. Hence each player will respond (in an optimal, cost minimizing manner) to the behavior of the mass population and not to behaviors of individual players. Validity of this property will be established later in Section \ref{tsec3} of the paper, but in anticipation of this, we first introduce the quantity
\begin{equation}\label{empiricalprocess}
m^n_{t}=\frac{1}{n}\sum_{j=1}^n \delta_{x_j^n(t)},
\end{equation}
as an empirical measure of the collection of states of the players, where $\delta$ is a Dirac measure on the state space. This enables us to introduce the long-term cost function of Player $j$ (to be minimized by him) in terms of only the self variables ($x_j$ and $u_j$) and $m_t^n, t\geq0$, where the latter can be viewed as an exogenous process (not directly influenced by Player $j$). But we first introduce a mean-field representation of the dynamics (S), which uses $m^n_t$ and will be used in the description of the cost. 

\subsection{Mean-field representation}\label{tsec2}
The system (S) can be written into a measure representation using the formula
$$\int \phi(w) \left[\sum_{i=1}^n \bar{\omega}_i \delta_{x_i}\right](dw)=\sum_{i=1}^n\bar{\omega}_i\phi(x_i),$$
where $\delta_{z}, z\in \mathcal{X}$ is a Dirac measure concentrated at $z$, $\phi$ is a measurable bounded function defined on the state space and $\bar{\omega}_i\in\mathbb{R}$.
Then, the system (S) reduces to the system
\begin{eqnarray}
\nonumber dx_j^n(t) &=&\left(\int_w f(t, x_j^n(t),u^n_j(t),w) \left[\frac{1}{n}\sum_{i=1}^n\delta_{x_i^n(t)}\right](dw)\right) dt\\ 
\nonumber &+& \sqrt{\epsilon}\left(\int_w  \sigma(t, x_j^n(t),u^n_j(t),w)\left[\frac{1}{n}\sum_{i=1}^n\delta_{x_i^n(t)}\right](dw)\right) d\mathbb{B}_{j}(t),\\
\nonumber x_j^n(0)&=&x_{j,0}\in \mathbb{R}^{k},\ k\geq 1,
j\in\{1,\ldots,n\},
\end{eqnarray}
which, by (\ref{empiricalprocess}), is equivalent to the following system (SM):
\begin{equation}
 \tag{SM} 
\begin{array}{lll}
\nonumber \displaystyle dx_j^n(t)&=&\left(\displaystyle\int_w f(t, x_j^n(t),u^n_j(t),w) m^n_{t}(dw)\right) dt \\  \nonumber && +\sqrt{\epsilon}\left(\displaystyle\int_w  \sigma(t, x_j^n(t),u^n_j(t),w)m^n_{t}(dw)\right) d\mathbb{B}_{j}(t),\\ \nonumber
x_j^n(0)&=&x_{j,0}\in \mathbb{R}^{k},\ k\geq 1,
j\in\{1,\ldots,n\}.
\end{array} \end{equation}
 The above representation of the system (SM) can be seen as a controlled interacting particles representation of a macroscopic McKean-Vlasov equation where $m^n_t$ represents the discrete density of the population.
 {\color{black}
 Next, we address the mean field convergence of the population profile process $m^n.$ To do so, we introduce the key notion of indistinguishability.
 \begin{defi}[Indistinguishability]
 We say that a family of processes $(x^n_1,x^n_2,\ldots, x^n_n)$ is indistinguishable (or exchangeable) if the law of $x^n$ is invariant by permutation over the index set $\{1,\ldots,n\}.$
 \end{defi}
 The solution  of (S) obtained under fixed control $u(\cdot)$ generates indistinguishable processes.
 For any permutation $\pi$ over $\{1, 2, \ldots, n\}$, one has
$\mathcal{L}(x^n_{j_1},\ldots,x^n_{j_n})=\mathcal{L}(x^n_{\pi(j_1)},\ldots,x^n_{\pi(j_n)}),\ $ where
$\mathcal{L}(X)$ denotes the law of the random variable $X.$  For indistinguishable (exchangeable)  processes, the convergence of the empirical measure has been widely studied (see \cite{Tanabe} and the references therein).
}
To preserve this property for the controlled system we restrict ourselves to admissible homogeneous controls. Then, the mean field convergence is equivalent  to the existence of a random measure $\mu$ such that the system is $\mu-$chaotic, i.e., $$
\lim_n \int \prod_{l=1}^L \phi_l({x}^n_{j_l})\mu^n(dx^n)=\prod_{l=1}^L \left(\int \phi_l d\mu\right),
$$ for any fixed natural number $L\geq 2$ and a collection of measurable bounded functions $\{\phi_l\}_{1\leq l\leq L}$ defined over the state space $\mathcal{X}.$
Following the indistinguishability property, one has that
 the law of $x^n_j=(x_j^n(t),\ t\geq 0)$ is $\mathbb{E}[m^n].$
 The same result is obtained by proving the weak convergence of the individual state dynamics to a macroscopic McKean-Vlasov equation (see later Proposition  \ref{prottt1}). Then, when the initial states are i.i.d. and given some homogeneous control actions $u,$ the solution of the state dynamics generates an indistinguishable random process  and  the weak convergence of the population profile process $m^n$ to $\mu$ is equivalent to the $\mu-$chaoticity.
For general results on  mean-field convergence of controlled stochastic differential equations, we refer to \cite{huang2011}.
  These processes depend implicitly on the strategies used by the players. Note that  an admissible control law $\bar{\gamma}$ may depend on time $t$, the value of the individual state $x_j(t)$ and the mean-field process $m_t$.
 The weak convergence of the process $m^n$ implies the weak convergence of its marginal $m^n_t$  and one can characterize the distribution of $m_t$ by the Fokker-Planck-Kolmogorov (FPK) equation:
 \begin{flushleft}
$\partial_t m_t+D^1_x\left(m_t\displaystyle\int_w f(t, x,u(t),w)m_t(dw)\right)$
\end{flushleft}
\begin{eqnarray} \label{fpk}
&=&\frac{\epsilon}{2}D^2_{xx}\left(m_t\left(\int_w \sigma'(t, x,u(t),w)m_t(dw)\right)\cdot\left(\int_w \sigma(t, x,u(t),w)m_t(dw)\right)\right).
\end{eqnarray}

Here $f(\cdot)\in\mathbb{R}^k,$ which we denote by $(f_{k'}(\cdot))_{1\leq k'\leq k},$ where $f_{k'}$ is scalar. We let
$$\underline{\sigma}[t, x,u(t),m_t]:=\int_w \sigma(t, x,u(t),w)m_t(dw), $$ $\Gamma(\cdot):=\underline{\sigma}(\cdot)\underline{\sigma}'(\cdot)$ is a square matrix with dimension $k\times k.$ The term $D^1_x(\cdot)$ denotes
$$\sum_{k'=1}^k \frac{\partial}{\partial x_{k'}}\left(m_t\int_w f_{k'}(t, x,u(t),w)m_t(dw)\right),$$
and the last term on $D^2_{xx}(\cdot)$ is
$$
\sum_{k''=1}^k\sum_{k'=1}^k \frac{\partial^2}{\partial x_{k'}\partial x_{k''}}\left(m_t\Gamma_{k'k''}(\cdot)\right).
$$
In the one-dimensional case, the terms $D^1,D^2$ reduce to the divergence ``$\textrm{div}$" and the Laplacian operator $\Delta$,  respectively.

It is important to note that the existence of a unique rest point (distribution) in FPK does not automatically imply that the mean-field converges to the rest point when $t$ goes to infinity.
This is because the rest point may not be stable.
\begin{rem}
In mathematical physics, convergence to an independent and identically distributed system is sometimes referred to as \textit{chaoticity} \cite{Sznitman,Tanabe,graham}, and
the fact that chaoticity at the initial time implies chaoticity at further times is called propagation of chaos. However in our setting the chaoticity property needs to be studied together with the controls of the players.
In general the chaoticity property may not hold. One particular case should be mentioned, which is when the rest point $m^*$ is related to the $\delta_{m^*}-$ chaoticity. If the mean-field dynamics has a unique global attractor $m^*$, then the propagation  of chaos property holds for the measure $\delta_{m^*}.$ Beyond this particular case, one may have multiple rest points but also the double limit, $\lim_n \lim_t m^n_t$  may differ from  the one when the order is swapped, $\lim_t \lim_n m^n_t$  leading a non-commutative diagram. Thus, a deep study of the underlying dynamical system is required if one wants to analyze a performance metric for a stationary regime.  A counterexample   of non-commutativity of the double limit is provided in \cite{meanfieldlectures}.\ged
\end{rem}

\subsection{Cost Function}\label{costfunction}
We now introduce the cost functions for the differential game. Risk-sensitive behaviors can be captured by cost functions which exponentiate loss functions before the expectation operator. For each $t\in[0, T]$, and $m_t^n, x_j$ initialized at a generic feasible pair $\underline{m}, \underline{x}$  at $t$, the risk-sensitive cost function for  Player $j$ is given by
\begin{eqnarray}\label{RSCF}
L(\bar{\gamma}_j,m^n_{[t, T]};t, \underline{x},\underline{m})=
\delta\log\mathbb{E}\left( e^{\frac{1}{\delta}[\displaystyle g(x_T)+\int_t^T c (s, x_j^n(s),u_j^n(s),m^n(s))\ ds]}\ \Bigg| \ x_j(t)=\underline{x},m^n_t=\underline{m} \right),
\end{eqnarray}
where $c(\cdot)$ is the instantaneous cost at time $s$; $g(\cdot)$ is the terminal cost; $\delta>0$ is the risk-sensitivity index;
$m^n_{[t, T]}$ denotes the process $\{m^n_s, t\leq s\leq T\}$; and $u_j^n(s) = \bar{\gamma}_j(s, x_j^n(s), m^n(s)),$ with $\bar{\gamma}_j\in \bar{\Gamma}_j$. Note that because of the symmetry assumption across players, the cost function of Player $j$ is not indexed by $j$, since it is in the same structural form for all players. This is still a game problem (and not a team problem), however, because each such cost function depends only on the self variables (indexed by $j$ for Player $j$) as well as the common population variable $m^n$.

We assume the following standard conditions on $c$ and $g$.
\begin{itemize}
\item[(vi)] $c$ is $C^1$ in $(t, x, u, m)$; $g$ is $C^2$ in $x$; $c, g$ are non-negative;
\item[(vii)]  $c, \partial_x c, g, \partial_x g$ are uniformly bounded.
\end{itemize}

 The cost function (\ref{RSCF}) is called the risk-sensitive cost functional or the exponentiated integral cost, {\color{black} which measures risk-sensitivity for the long-run and not at each instant of time (see \cite{jacobson,whittle,bensoussan,Basar99}).
}
  {\color{black}
  We note that the McKean-Vlasov mean field game considered here differs from the model in \cite{MRP}; specifically, in this paper, the volatility term in (SM) is a function of state, control and the mean field, and further, the cost functional is of the risk-sensitive type.
 }

 {\color{black}
 \begin{rem}[Connection with mean-variance cost]
 Consider the function $ c^{\lambda}:\ \lambda\longmapsto\frac{1}{\lambda} \log(\mathbb{E}e^{\lambda C}).$
  It is obvious that the risk-sensitive cost $c^{\lambda}$ takes into consideration all the moments of the cost $C$, and not only its mean value.
  Around zero, the Taylor expansion of $c^\lambda$ is given by
$$ c^{\lambda}\underbrace{\approx}_{\lambda\sim 0} \mathbb{E}(C) +\frac{\lambda}{2}\textrm{var}(C)+o(\lambda),
  $$ where the important terms are the mean cost and the variance of the cost for small $\lambda.$ Hence risk-sensitive cost entails a weighted sum of the mean and variance of the cost, to some level of approximation.
 \end{rem}
 }


With the dynamics (SM) and cost functionals as introduced, we seek an individual state-feedback non-cooperative Nash equilibrium $\{\bar{\gamma}^*_i, i\in\{1, \cdots, n\}\}$, satisfying the set of inequalities
\begin{equation}
L(\bar{\gamma}^*_j, m^n_{[0, T]}; 0, x_{j,0}, \underline{m}) \leq L(\bar{\gamma}_j, \tilde{m}^{n,j}_{[0, T]}; 0, x_{j,0}, \underline{m}),
\end{equation}
for all $\bar{\gamma}_j\in\bar{\Gamma}_j, j\in\{1,2,\cdots, n\}$, where $m^n[0, T]$ is generated by the $\bar{\gamma}_j^*$'s, and $\tilde{m}^{n,j}_{[0,T]}$ by $(\bar{\gamma}, \bar{\gamma}_{-j}^*)$, $\bar{\gamma}_{-j}^*=\{\bar{\gamma}_i^*, i = 1, 2, \cdots, n, i \neq j\}$;
 $u^*_j$ and $u_j$ are control actions generated by control laws $\bar{\gamma}^*_j$ and $\bar{\gamma}_j$, respectively, i.e., $u^*_j=\bar{\gamma}^*_j(t, x_j)$ and $u_j=\bar{\gamma}_j(t, x_j)$; $m^n_t=m^n_t[u^*]$ laws are given by  forward FPK equation under the strategy $\bar{\gamma}^*,$ and $\tilde{m}^{n,j}_t=\tilde{m}^{n,j}_t[u_j,u_{-j}^*]$ is the induced measure under the strategy $(\bar{\gamma}_j,\bar{\gamma}_{-j}^*).$

A more stringent equilibrium solution concept is that of
 strongly time-consistent individual state-feedback Nash equilibrium satisfying,
\begin{equation}
L(\bar{\gamma}^*_j, m^n_{[t, T]}; t, x_j, \underline{m}) \leq L(\bar{\gamma}_j, \tilde{m}^{n,j}_{[t, T]}; t, x_j, \underline{m}),
\end{equation}
for all $x_j\in\mathcal{X}$, $t\in[0, T)$, $\bar{\gamma}_j\in\bar{\Gamma}_j, j\in\{1,2,\cdots, n\}.$

Note that the two measures $m^n_t$ and $\tilde{m}^{n,j}_t$ differ only in the component $j$ and  have a common term which is
 $\frac{1}{n}\sum_{j'\neq j}\delta_{x^n_{j'}(t)}$, which converges in distribution to some measure with a distribution that is a solution of the forward PFK partial differential equation.

\section{Risk-sensitive best response to mean-field and equilibria} \label{tsec3}
In this section, we present the risk-sensitive mean-field  results. We first provide an overview of  the mean-field (feedback) best response  for a given mean-field trajectory $m^n=(m^n(s),\ s\geq 0).$
A mean-field best-response strategy of a generic Player $j$ to a given mean field $m_t^n$ is a measurable mapping $\bar{\gamma}^*_j$ satisfying: $\ \forall \ \bar{\gamma}_j\in\bar{{\Gamma}}_j$, with $x_j$ and $m^n_t$ initialized at $x_{j,0},\underline{m}$, respectively,
$$
L(\bar{\gamma}_j^*, m^n_{[0, T]},0,x_{j,0},\underline{m})\leq L(\gamma_j,m^n_{[0,T]} ,0,x_{j,0},\underline{m}).
$$
where law of $m^n_t$ is given by the forward FPK equation in the whole space $\mathcal{X}^n$, and is an exogenous process. 
Let  $v^n(t,x_j,\underline{m})=\inf_{u_j} L({u}_j,m^n_{[0, T]},t,x_j,\underline{m}).$
The next proposition establishes the risk-sensitive Hamilton-Jacobi-Bellman (HJB) equation of the risk-sensitive cost function satisfied by a smooth optimal value function of a generic player. The main difference from the standard HJB equation is the presence of the term  $\frac{\epsilon}{2\delta}\parallel \sigma \partial_{x_j} v^n\parallel^2.$

\begin{prop}\label{lemm1}
Suppose that the trajectory of $m^n_t$ is given.
If $v^n$ is twice continuously differentiable, then $v^n$ is solution of the risk-sensitive HJB equation
\begin{eqnarray}\label{HJBFEprop1}
\nonumber\partial_tv^n+\inf_{u_j}\left\{  f \cdot\partial_{x_j}v^n +\frac{\epsilon}{2}{tr}(\sigma \sigma' \partial^2_{x_jx_j}v^n_j)+\frac{\epsilon}{2\delta}\parallel \sigma \partial_{x_j} v^n\parallel^2+c\right\}&=&0, \\
\nonumber v^n(T,x_j)&=&g(x_j).
\end{eqnarray}
Moreover, any strategy satisfying 
$$\bar{\gamma}^n_j(\cdot)\in
\arg\min_{u_j}\left\{ f \cdot\partial_{x_j}v^n +\frac{\epsilon}{2} {tr}(\sigma\sigma'\partial^2_{x_jx_j}v^n)+\frac{\epsilon}{2\delta}\parallel \sigma\partial_{x_j} v^n\parallel^2+c\right\},
$$ constitutes a best response strategy to the mean-field $m^n.$
\end{prop}
\begin{proof}[Proof of Proposition \ref{lemm1}]
For  feasible initial conditions $\underline{x}$ and $\underline{m}$, we define $$\phi^n(t,\underline{x},\underline{m}):=\inf_{{u}^n_j}  \mathbb{E}\left( e^{\frac{1}{\delta}[g(x_T)+\int_t^T c(s, x^n(s),u_{j}(t),m^n_s)\ ds]}\  | \ x_j(t)=\underline{x},m^n_t=\underline{m} \right).$$
It is clear that
$
v^n(t,x_j,\underline{m})=\inf L =\delta \log \phi^n(t,x_j,\underline{m}).
$
Under the regularity assumptions of Section \ref{tsec1}, the function $\phi^n$ is $C^1$ in $t$ and $C^2$ in $x.$ Using It\^o's formula,
$$
d\phi^n(t,x_j)=[\partial_t\phi ^n(t,x_j)+ f \cdot\partial_{x_j}\phi^n +\frac{\epsilon}{2}\textrm{tr}(\sigma\sigma' \partial^2_{x_jx_j}\phi^n)] dt.
$$
Using the Ito-Dynkin's formula (see  \cite{oksendal, bensoussan,panbasar}), the dynamic optimization yields
$$\inf_{\bar{u}_j} \{ d\phi^n+\frac{1}{\delta}c\phi^n dt \}=0.
$$
Thus, one obtains
\begin{eqnarray}
\nonumber \partial_t\phi^n+\inf_{u_j}\left\{  f \cdot\partial_{x_j}\phi^n +\frac{\epsilon}{2}\textrm{tr}(\sigma\sigma'\partial^2_{xx}\phi^n)+\frac{1}{\delta}c \phi ^n\right\}&=&0,\\
\nonumber
\phi^n(T,x_j)&=&e^{\frac{1}{\delta}g(x_j)}.
\end{eqnarray}
To establish the connection with the risk-sensitive cost value,
we use the relation $\phi^n=e^{\frac{1}{\delta}v^n}$.
One can compute the partial derivatives:
$$
\partial_t \phi^n=\left( \partial_t v^n\right) \frac{1}{\delta}\phi^n
, \ \
\partial_{x_j} \phi^n=\left( \partial_{x_j} v^n \right) \frac{1}{\delta}\phi^n,
$$
and
$$
\partial^2_{x_jx_j} \phi^n=\left( \partial^2_{x_jx_j} v^n\right) \frac{1}{\delta}\phi^n+\frac{1}{\delta^2}
\left( \partial_{x_j} v^n\right)' \left( \partial_{x_j} v^n\right)\phi^n,
$$
where the latter immediately yields 
$$\textrm{tr}(\partial^2_{x_jx_j} \phi^n \sigma \sigma')=\textrm{tr}(\partial^2_{x_jx_j} v^n\sigma \sigma')\frac{1}{\delta}\phi^n+\frac{1}{\delta^2}\parallel \sigma\partial_{x_j} v^n\parallel^2\phi^n.$$
Combining together and dividing by $\phi^n/\delta,$ we arrive at the HJB equation (\ref{HJBFEprop1}).

\end{proof}
\begin{rem}
Let us introduce the Hamiltonian $H$ as
$$H(t, x,\tilde{p},\tilde{M})=\inf_{u}\left\{ \tilde{p} \cdot f+\frac{\epsilon}{2} {tr}(\sigma\sigma' \tilde{M})+\frac{\epsilon}{2\delta}\parallel \sigma\tilde{p}\parallel^2+c\right\},
$$ for a vector $\tilde{p}$ and a matrix $\tilde{M}$ which is the same as the Hessian of $v^n.$

If $\sigma$ does not depend on the control, then the above expression reduces to
$$
\inf_{u}\{\tilde{p} \cdot f+ c \}+ \frac{\epsilon}{2} {tr}(\sigma\sigma' \tilde{M})+\frac{\epsilon}{2\delta}\parallel \sigma\tilde{p}\parallel^2,
$$ and the term to be minimized is
$H^2(t, x,\tilde{p},\tilde{M})=\inf_{u}\{\tilde{p}\cdot f+ c \},$ which is related to the Legendre-Fenchel transform for linear dynamics, i.e., the case where $f$ is linear in the control $u.$

In that case, $$\partial_{\tilde{p}}H^2(t, x,\tilde{p},\tilde{M})=\alpha u^*$$  for some non-singular
$\alpha$ of proper dimension. This says that the derivative of the modified Hamiltonian is related to the optimal feedback control. Now, for non-linear drift $f$ the same technique can be used but the function $f$ needs to be inverted to obtain a generic closed form expression the optimal feedback control and is given by
$$
u^*_{j}=\tilde{g}^{-1}(\partial_{\tilde{p}}H^2(t, x,\tilde{p},\tilde{M})),
$$ where $\tilde{g}^{-1}$ is the inverse of the map $$ u\longmapsto f(t, x,u,m).$$
This generic expression of the optimal control will play an important role in non-linear McKean-Vlasov mean field games.
\end{rem}

The next proposition provides the best-response control to the affine-quadratic in $u$-exponentiated cost-Gaussian mean-field game, and the proposition that follows that deals with the case of affine-quadratic in both $u$ and $x$.
\begin{prop}\label{Prop2} Suppose $\sigma(t,x)=\sigma(t)$ and \vspace{-3mm}
\begin{eqnarray}
\nonumber f(t,x_j,u_j,m)&=&\bar{f}(t,x_j,m)+B(t,x_j,m)u_j,\\
\nonumber c(t, x_j,u_j,m)&=&\bar{c}(t,x_j,m)+\parallel u_j\parallel^2.
\end{eqnarray}
Then, the best-response control of Player $j$ is
$
\bar{\gamma}_j^{n,*}=-\frac{1}{2}B\partial_{x_j}v^n.
$
\end{prop}
\begin{proof}
Following Proposition \ref{lemm1}, we know
$$
\bar{u}^{n,*}_j = \bar{\gamma}^{n,*}_j(\cdot)  \in \arg\min_{u_j}\{ c(t, x_j(t),u_j(t),m_t)+  f(t, x_j(t),u_j,m_t)\cdot\partial_{x_j}v^n \}.
$$
With the assumptions on $\sigma, f, c, g$, the condition reduces to
$$
\arg\min_{u_j} \left\{ [\bar{f}+Bu_j]\partial_{x_j}v^n +\bar{c}+\parallel u_j\parallel^2\right\}.
$$
and hence, we obtain
$\bar{\gamma}_j^{n,*}=-\frac{1}{2}B\partial_{x_j}v^n$ by convexity and coercivity of the mapping
$u_j \longmapsto [\bar{f}+Bu_j]\partial_{x_j}v^n +\bar{c}+\parallel u_j\parallel^2.$
\end{proof}

\begin{prop}[Explicit optimal control and cost, \cite{Basar99}] \label{propott31}~\\
Consider the risk-sensitive mean-field stochastic game described in Proposition \ref{Prop2} with
 $\bar{f} = A(t) x$, $B$ a constant matrix, $ c=x'Q(t) x,\ Q(t)\geq 0,\ g(x)=x'Q_Tx, Q_T\geq0,$ where the symmetric matrix $Q(\cdot)$ is continuous. Then, the solution to HJB equation in Proposition \ref{lemm1} (whenever it exists) is given by
$
v^n(t,x)=x'Z(t)x+\epsilon\int_t^T {tr}(Z(s)\sigma\sigma')\ ds.
$
where $Z(s)$ is the nonnegative definite solution of the generalized Riccati differential equation
$$\dot{Z}+A'Z+ZA+Q-Z\left(BB'-\frac{1}{\rho^2}\sigma\sigma'\right)Z=0,\ Z(T)=Q_T,$$
where $\rho=(\frac{\delta}{2\epsilon})^{1/2}$
and the optimal response strategy is
\begin{equation} \label{optimal}
u^*_j(t)=\bar{\gamma}^*_j(\cdot)=-B'Zx.
\end{equation}
\end{prop}

Using  Proposition \ref{propott31}, one has the following result for any given trajectory $(m^n_t)_{t\geq 0},$ which enters the cost function in a particular way.

\begin{prop}
If $c$ is in the form $c=x'(Q(t)-\Lambda(t, m^n_t))x$, where $\Lambda$ is symmetric and continuous in $(t, m)$, then the generalized Riccati equation becomes
$$
\dot{Z}^*+A'Z^*+Z^*A+Q-\Lambda(t,m^n_t)-Z^*\left(BB'-\frac{1}{\rho^2}\sigma\sigma'\right)Z^*=0, Z^*(T)= Q_T,
$$
and
$$
v^n(t,x)=x'Z^*x+\epsilon\int_t^T {tr}(Z^*(s)\sigma\sigma')\ ds.
$$
\end{prop}

\subsection{Macroscopic McKean-Vlasov equation} \label{secmacrotr}

Since  the controls used by the players  influence the mean-field limit via the state dynamics, we need to characterize the evolution of the mean-field limit as a function of the controls. The law of $m_t$ is the solution of the Fokker-Planck-Kolmogorov equation given by (\ref{fpk}) and the individual state dynamics follows the so-called macroscopic McKean-Vlasov equation
\begin{eqnarray}
d\bar{x}_j(t)=\left(\int_w f(t, \bar{x}_j(t),u_j^*(t),w) m_{t}(dw)\right) dt + \sqrt{\epsilon}\left(\int_w  \sigma(t, \bar{x}_j(t),u_j^*(t),w)m_{t}(dw)\right) d\mathbb{B}_{j}(t).
\end{eqnarray}

In order to obtain an error bound, we introduce the following notion: Given two measures $\mu$ and $\nu$ the Monge-Kontorovich metric (also called Wasserstein metric) between $\mu$ and $\nu$ is
$$\mathcal{W}_1(\mu,\nu)=\inf_{X\sim \mu, Y \sim \nu}\mathbb{E}|X-Y|.$$
In other words, let $E(\mu, \nu)$ be the set of probability measures $\mathbb{P}$ on the product space
such that the image of $\mathbb{P}$ under the projection on the first argument (resp. on the
second argument) is $\mu$ (resp. $\nu$). Then,
\begin{eqnarray}\mathcal{W}_1(\mu,\nu)=\inf_{\mathbb{P}\in E(\mu, \nu)}\int \int |z-z'| \mathbb{P}(dz,dz').
\end{eqnarray}
This is known indeed as a distance (it can be checked that the separation, the triangle inequality and positivity properties  are satisfied) and it metricizes the weak topology.

\begin{prop} \label{prottt1}
Under the conditions (i)-(vii), the following holds: For any $t,$ if the control law $\gamma^*_j(\cdot)$ is used, then there exists $\tilde{y}_t>0$ such that
$$\mathbb{E}\left( \parallel x^n_j(t)-\tilde{x}_j(t)\parallel\right)\leq \frac{\tilde{y}_t}{\sqrt{n}}.$$
Moreover, for any $T<\infty,$ there exists $C_T>0$ such that
\begin{eqnarray}\mathcal{W}_1\left(\mathcal{L}((x^n_j(t))_{t\in [0,T]}), \mathcal{L}((\tilde{x}_j(t))_{t\in [0,T]})\right)\leq \frac{C_T}{\sqrt{n}},
 \end{eqnarray} where $\mathcal{L}(X_t)$ denotes the law of the random variable $X_t$.
\end{prop}
The last inequality  says that the error bound is at most of $O(\frac{1}{\sqrt{n}})$ for any fixed compact interval.
 The proof of this assertion follows the following steps:
Let $x^n_{j}(t)$ and $\tilde{x}_j(t)$ be the solutions of the two SDEs with initial gap less than $\frac{1}{\sqrt{n}}.$ Then, we take the difference between the two solutions. In a second step,
 use triangle inequality of norms and take the expectation. Gronwall inequality allows one to complete the proof. A detailed proof is provided in the Appendix.


\subsubsection{Risk-sensitive mean-field cost}
Based on the fact that $m^n_t$ converges weakly to $m_t$ under the admissible controls $(u^n_j(s),\ s\geq 0)\longrightarrow (u_j(s),\ s\geq 0)$ when $n$ goes to infinity, one can show the weak convergence of the risk-sensitive cost function (\ref{RSCF}) under the regularity conditions (vi) and (vii) on functions $c$ and $g$, i.e., as $n\rightarrow \infty$,
\begin{eqnarray}
\nonumber L(\bar{\gamma}_j,m^n_{[t, T]};t, \underline{x},\underline{m}) &\rightarrow& L(u_j,m_{[t, T]},t, \underline{x},\underline{m})\\
\nonumber & & =
\delta\log \mathbb{E}\left( e^{\frac{1}{\delta}[g(x_j(T))+\int_t^T c(s, x_{j}(s),u_{j}(s),m_s)\ ds]} \bigg| \ x_j(t)=\underline{x},m_t=\underline{m}\right).
\end{eqnarray}

Based on this limiting cost, we can construct the best  response to mean field in the limit. Given $\{m_s\}_{s\in[t, T]}$, we minimize $L(u_j,m_{[t, T]};t,x,m)$ subject to the state-dynamics constraints.
\subsection{Fixed-point problem}
{\color{black}We now define the mean field equilibrium problem as the following fixed-point problem.
\begin{defi} The mean field equilibrium problem (P) is one where 
each player solves the optimal control problem, i.e.,
$${\inf}_{u_j}\ \delta\log \mathbb{E}\left( e^{\frac{1}{\delta}[g(x_j(T))+\int_t^T c(s, x_{j}(s),u_{j}(s),m_s^*)\ ds]} \bigg| \ x_j(t)=\underline{x},m_t=\underline{m}\right),$$
subject to the dynamics of $x_{j}(t)$  given by the dynamics in Section \ref{secmacrotr}, where the mean field $m_t$ is replaced by $m^*_t$
and $\bar{m}_t^*$ is the mean of the optimal mean field trajectory.
The optimal feedback control $u_j^*[t,x,m^*]$ depends on $m^*$, and $m^*$ is the mean field reproduced by all the $u_j^*$, i.e., $m^*_t=m[t,u^*]$ solution of the Fokker-Planck-Kolmogorov forward equation (\ref{fpk}). The equilibrium is called an individual feedback mean field equilibrium if every player adopts an individual state-feedback strategy.
\end{defi}
}
Note that this problem differs from the risk-sensitive mean field stochastic optimal control problem where  the objective is
$$\delta\log \mathbb{E}\left( e^{\frac{1}{\delta}[g(x_j(T))+\int_t^T c(s, x_{j}(s),u_{j}(s),m_s[u])\ ds]} \bigg| \ x_j(t)=\underline{x},m_t=\underline{m}\right),$$ with $m_s[u]$ the distribution of the state dynamics $x_j(s)$ driven by the control $u_j.$

\subsection{Risk-sensitive FPK-McV equations} \label{counter1}
The  regular solutions to problem (P) introduced above are solutions to { HJB backward equation combined with FPK equation and macroscopic McKean-Vlasov version of the limiting individual dynamics}, i.e.,
\begin{eqnarray}
\nonumber d{x}_j(t)&=&\left(\int_w f(t, x_j(t),u_j^*(t),w) m_{t}(dw)\right) dt \\ 
\nonumber &+& \sqrt{\epsilon}\left(\int_w  \sigma(t, x_j(t),u_j^*(t),w)m_{t}(dw)\right) d\mathbb{B}_{j}(t),\\
\nonumber &&x_j(0)=x_{j,0}=x
\end{eqnarray}
\begin{eqnarray}
\nonumber 0&=&\partial_tv+\inf_{u_j}\left\{  f \cdot\partial_{x}v+\frac{\epsilon}{2}{tr}(\sigma\sigma'\partial^2_{xx}v)+\frac{\epsilon}{2\delta}\parallel \sigma\partial_x v\parallel^2+c\right\},\\
\nonumber && x_j:=x; \ \  v(T,x)=g(x)
\end{eqnarray}
\begin{eqnarray}
\nonumber \partial_t m_t&=&-D_x^1\left(m_t\int_w f(t,x,u^*,w)m_t(dw)\right)\\ 
\nonumber && +\frac{\epsilon}{2}D^2_{xx}\left(m_t\left(\int_w \sigma'(t, x,u^*,w)m_t(dw)\right) \right. \left.\cdot \left(\int_w \sigma(t,x,u^*,w)m_t(dw)\right)\right)\\ 
\nonumber && m_0(\cdot)\ \mbox{fixed}.
\end{eqnarray}

Then, the question of existence of a solution to the above system arises. This is a backward-forward system. Very little is known about the existence of a solution to such a system. In general, a solution may not exist as the following example demonstrates. 
\subsection{Non-existence of solution to backward-forward boundary value problems}\label{counter2}
There are many examples of systems of backward-forward equations which do not admit
solutions. As a very simple  example from \cite{zhou}, consider the system:
$$\dot{v}=m,\ \dot{m}=-v, m(0)=m_0; v(T)=-m_T.$$

It is obvious that the coefficients of this pair of backward-forward differential equations  are
all uniformly Lipschitz. However,   depending on $T$, this may not be
solvable for $m_0\neq 0.$
We can easily show that for $T = k\pi + 3\pi/4$ ($k$, a nonnegative integer), the
above two-point boundary value problem does not admit a solution for any
$m_0\neq 0$  and it admits infinitely many solutions for $m_0=0.$

  Following the same ideas, one can show that the system of stochastic differential equations (SDEs) $$d{v}=m dt+\sigma d\mathbb{B}(t),\ d{m}=-v dt+\nu d\mathbb{B}(t),$$ where $\mathbb{B}(t)$ is the standard Brownian motion in $\mathbb{R}$. With the initial conditions: $$m(0)=m_0\neq 0; v(T)=-m_T,$$ and $T=7\pi/4$, the system of SDEs has no solution.

  This example shows us that the system needs to be normalized and the boundary conditions will have to be well posed. In view of this,  we will introduce the notion of reduced mean field system in Section \ref{uni} to establish the existence of equilibrium for a specific class of risk-sensitive games.

\subsection{Risk-sensitive mean-field equilibria}
\begin{thm} \label{thmt1}
Consider a risk-sensitive mean-field stochastic differential game as formulated above. Assume that  $\sigma =\sigma(t)$ and there exists a unique pair $(u^*, m^*)$ such that

(i) The coupled backward-forward PDEs
\begin{eqnarray}
\nonumber \partial_tv^*+\inf_{u_j}\left\{  f^*\cdot\partial_{x}v^* +\frac{\epsilon}{2} {tr}(\sigma\sigma'\partial^2_{xx}v^*)\right. \left.+\frac{\epsilon}{2\delta}\parallel \sigma\partial_x v\parallel^2+c^*\right\}&=&0,\\
 \nonumber v(T,x)=g(x),\ m^*_0(x)\ {\textrm{fixed}}. 
 \end{eqnarray}
 \begin{eqnarray}
\nonumber \partial_t m^*_t&+&D^1_x\left(m^*_t\int_w f^*(t,x,u^*,w)m^*_t(dw)\right)  \\
\nonumber  &=&\frac{\epsilon}{2}D^2_{xx}\left(m^*_t\left(\int_w \sigma'm^*_t(dw)\right)\left(\int_w \sigma m^*_t(dw)\right)\right) 
\end{eqnarray} 
admit a pair a bounded nonnegative solutions $v^*,m^*$; and

{\color{black} (ii) $u^*$ minimizes the Hamiltonian, i.e., $f(t, x,u,m^*) \cdot \partial_x v^*+c(t, x,u,m^*)$.}

Under these conditions, the pair $(u^*,m^*)$ is a strongly time-consistent mean-field equilibrium and
$L(t,u^*,m^*)=v^*.$
In addition, if $c=x'(Q(t)-\Lambda_t(m^n_t))x$ where $\Lambda(t,\cdot)$ is a measurable symmetric matrix-valued function, then any convergent subsequence of optimal control laws $\bar{\gamma}^{\alpha(n)}_j$ leads to a best strategy for $m.$

\end{thm}
\begin{proof} See the Appendix.
\end{proof}
\begin{rem}
This result can be extended to finitely multiple classes of players (see~\cite{HuangCIS, bardi, lasry1} for discussions).
To do so, consider a finite number of classes indexed by $\theta\in\Theta.$ The individual dynamics are indexed by $\theta,$ i.e. the function $f$ becomes $f_{\theta}$ and $\sigma$ becomes $\sigma_{\theta}.$ This means that the indistinguishability property is not satisfied anymore. The law depends on $\theta$ (it is not invariant by permutation of index). However, the invariance property  holds within each class. This allows us to establish a weak convergence of the individual dynamics of each generic player for each class, and we obtain $\tilde{x}_{\theta}(t).$ The multi-class mean-field equilibrium will be defined  by a system for each class and the classes are interdependent via the mean field and the value functions per class.\ged
\end{rem}
\subsubsection*{Limiting behavior with respect to $\epsilon$ }
We scale the parameters $\delta,\epsilon$ and $\rho$ such that $\delta=2\epsilon\rho^2.$
The PDE given in Proposition \ref{lemm1} becomes
$$  \partial_tv+\inf_{u}\left\{  f^*\cdot\partial_{x}v
+\frac{\epsilon}{2}{tr}(\sigma\sigma'\partial^2_{xx}v)+\frac{1}{4\rho^2}\parallel \sigma\partial_{x} v\parallel^2+c^*\right\}=0,\
 v(T,x)=g(x).$$
When the parameter $\epsilon$ goes to zero, one arrives at a deterministic PDE. This situation captures the large deviation limit:
$$
\partial_tv+\inf_{u}\left\{  f^*\cdot\partial_{x}v +\frac{1}{4\rho^2}\parallel \sigma\partial_{x} v\parallel^2+c^*\right\}=0,\ v(T,x)=g(x).$$

{\color{black}
\subsection{Equivalent stochastic mean-field problem}
In this subsection, we formulate an equivalent $(n+1)-$player game in which the state dynamics of the $n$ players are given by the system (ESM) as follows:
\begin{equation}
 \tag{ESM} 
\begin{array}{lll}
\nonumber dx_j^n(t)&=&\left(\displaystyle\int_w f(t,x_j^n(t),u^n_j(t),w) m^n_{t}(dw)  +\sigma \zeta (t) \right) dt+ \sqrt{\epsilon}  \sigma d\mathbb{B}_{j}(t),\\
\nonumber  x_j^n(0)&=&x_{j,0}\in \mathbb{R}^{k},\ k\geq 1,
j\in\{1,\ldots,n\},
\end{array} \end{equation}
where $\zeta (t)$ is the control parameter of the ``fictitious" $(n+1)-$th player. In parallel to (\ref{RSCF}), we define the risk-neutral cost function of the $n$ players as follows:
\begin{flushleft}
$\tilde{L}(\bar{\gamma}_j,\bar{\zeta},x^n_j,m_{[0,T]}^n; t,\underline{x},\underline{m})=$
\end{flushleft}
\begin{equation}\label{ESMRSCF}
\mathbb{E}\left( g(x^n_{j} (T))+\int_t^T c(s, x_{j}^n(s),u_{j}^n(s),m^n_s)\ ds -\rho^2\int_t^T \parallel \zeta (s)\parallel^2\ ds\  \bigg| x_j(t) =  \underline{x}, m^n_t =\underline{m}\right),
\end{equation}
where $\bar{\zeta}: [0, T]\times \mathbb{R}^k \rightarrow U_{n+1}$ is the individual feedback control strategy of the fictitious Player $n+1$ that yields an admissible control action $\zeta(t)$ in a set of feasible actions $U_{n+1}$.

Every player $j\in\{1,2,\ldots,n\}$ minimizes $\tilde{L}$ by taking the worst over the feedback strategy $\bar{\zeta}$ of player $n+1$ which is piecewise continuous in $t$ and Lipschitz in $x_j.$ We refer to this game described by (ESM) and (\ref{ESMRSCF}) as the {\it robust mean-field game}. In the following Proposition, we describe the connection between the mean-field risk-sensitive game problem described in (SM) and (\ref{RSCF}) and the robust mean-field game problem described in (ESM) and (\ref{ESMRSCF}),

\begin{prop}\label{proo}Under the regularity assumptions (i)-(vii), given a mean field $m^n_t$, the value functions of the risk-sensitive game and the robust game problems are identical, and the mean-field best-response control strategy of the risk-sensitive stochastic differential game is identical to the one for the corresponding robust mean-field game.
\end{prop}

\begin{proof}
 Let 
$\tilde{v}^n=\inf_{{u}_j}\sup_{{\zeta}}\tilde{L}({u}_j,{\zeta},x^n_j,m_{[0, T]}^n,t,x_j,\underline{m})$ denote the upper-value function associated with this  robust mean-field game. Then, under the regularity assumptions (i)-(vii), if $\tilde{v}^n$ is $C^1$ in $t$ and $C^2$ in $x$, it satisfies the Hamilton-Jacobi-Isaacs (HJI) equation
\begin{eqnarray}\label{H3}
\inf_{u}\sup_{\zeta}\left\{ \partial_t\tilde{v}^n_j+ (f+\sigma\zeta)\cdot \partial_{x_j}\tilde{v}^n+c-\rho^2\parallel \zeta \parallel^2 +\frac{\epsilon}{2}{tr}(\partial^2_{x_jx_j}\tilde{v}^n \sigma\sigma') \right\}&=&0, \\
\nonumber \tilde{v}^n(T,x_j)&=&g(x_j).
\end{eqnarray}
Note that (\ref{H3}) can be rewritten as $\inf_{u}\sup_{\zeta} H^3$, where $$H^3:=H+(\sigma\zeta)\cdot \partial_{x_j}\tilde{v}^n -\rho^2\parallel \zeta\parallel^2$$ is the Hamiltonian associated with this robust game.

Since the dependence on $u$ and $\zeta$ above are separable, the Isaacs condition (see \cite{basar1999dynamic}) holds, i.e., $$\inf_{u}\sup_{\zeta} H^3=\sup_{\zeta}\inf_{u} H^3$$ and hence the function $\tilde{v}^n_j$ satisfies the following after obtaining the best-response strategy for $\zeta$:
\begin{eqnarray}\label{pde2}
-\partial_t\tilde{v}^n&=& \inf_{u}\left\{ f\cdot\partial_{x_j}\tilde{v}^n+c+\frac{1}{4\rho^2}\parallel \sigma' \partial_{x_j}\tilde{v}^n\parallel^2
 +\frac{\epsilon}{2}{\textrm{tr}}(\partial^2_{x_jx_j}\tilde{v}^n \sigma\sigma')\right\}.\\
 \nonumber \tilde{v}^n(T,x_j)&=&g(x_j).
\end{eqnarray}
Note that the two PDEs, (\ref{pde2}) and the one given in Proposition~\ref{lemm1}, are identical with $\rho^2=\frac{\delta}{2\epsilon}$. Moreover, the optimal cost and the optimal control laws in the two problems are the same. 
\end{proof}
\begin{rem}
The FPK forward equation will have to be modified to include the control of fictitious player in the robust mean field game formulation accordingly by including the term $\sigma \zeta$ in (ESM). Hence the mean field equilibrium solutions to the two games are not necessarily identical.
\end{rem}
}

\section{ Linear state dynamics} \label{uni}
In this section, we analyze a specific class of risk-sensitive games where
  state  dynamics are linear and do not depend explicitly on the mean field.
We first state a related  result from \cite{pierrelouis,olivieruni} for the risk-neutral case.
\begin{thm}[\cite{pierrelouis}] \label{mainresultofPierreLouis}
Consider the reduced mean field system (rMFG):
\begin{eqnarray}
\nonumber \partial_xv+{H}(x,\nabla_xv, m_t(x))+\frac{\sigma^2}{2}\partial^2_{xx}v &=&0,\\
\nonumber\partial_xm_t+\textrm{div} (m_t\partial_p{H}(x,\nabla_xv, m_t(x))-\frac{\sigma^2}{2}\partial^2_{xx} \nonumber m_t&=&0,\\
\nonumber m_0(\cdot) \ \mbox{\textrm{fixed}},  v(T, \cdot) \ \mbox{\textrm{fixed}}, &&\\
\nonumber v,m \ \mbox{ \textrm {are 1-periodic.}}, &&\\
\nonumber x\in (0,1)^{d}:=\mathcal{X},
\end{eqnarray} 
where $H$ is the Legendre transform (with respect to the control) of the instantaneous cost function.

 Suppose that
$(x,p,z)\longmapsto H(x,p,z)$ is twice continuously differentiable with the respect to $(p,z)$ and for all $(x,p,z)\in \mathcal{X}\times \mathbb{R}^p\times \mathbb{R}_{+}^*,$
$$\left( \begin{array}{cc}
\partial^2_{pp}H(x,p,z) & \frac{1}{2}\partial_{pz}^2H(x,p,z)\\
\frac{1}{2}[\partial_{pz}^2H(x,p,z)]' & -\frac{1}{z}\partial_z H(x,p,z)
\end{array}\right)\succ 0
$$

Then, there exists at most one smooth solution to the (rMFG).
\end{thm}

\begin{rem}
We have a number of observations and notes.
 \begin{itemize}
   \item The Hamilitonian function $H$ in the result above requires a special structure. Instead of a direct dependence on the mean field distribution $m_t$, its dependence on the mean field is through the value of $m_t$ evaluated at state $x$. 
   \item
For  global dependence on $m,$ a sufficiency condition for uniqueness can be found in \cite{lasry1} for the case where the Hamiltonian is separable, i.e.,  $H(x,p,m)=\xi(x,p)+\tilde{f}(x,m)$ with $\tilde{f}$ monotone
in $m$ and $\xi$ strictly convex in $p.$\item
The solution of (rMFG) can be unique even if the above  conditions are violated. Further, the uniqueness condition is independent of the horizon of the game.
\item
For the linear-quadratic mean field case, it has been shown in \cite{bardi} that the normalized  system may have a unique i.i.d. solution or infinitely many solutions depending on the system parameters. See also \cite{beso} for recent analysis on risk-neutral linear-quadratic mean field games.
\end{itemize}\ged
\end{rem}
The next result provides the counterpart of   Theorem \ref{mainresultofPierreLouis}   in the risk-sensitive case.
It provides sufficient conditions for having at most one smooth solution in the risk-sensitive mean field system by exploiting the presence of the additive quadratic term (which is strictly convex in $p$).
\begin{thm} \label{secondmainresult} Consider the risk-sensitive (reduced) mean field system (RS-rMFG).
Let $\delta>0,$ and $H(x,p,z)$ be twice continuously differentiable in $(p,z)\in \mathbb{R}^d\times\mathbb{R}_{+},$ satisfying the following conditions:
\begin{itemize}
\item $H$ is strictly convex in $p,$
\item $H$ is decreasing in $z,$
\item $\left( -\frac{\partial_z H}{z}\right)\cdot\left( \partial^{2}_{pp}H \right)\succ (\partial^2_{pz}H-\frac{\epsilon\sigma^2}{2\delta}p/z)'\cdot(\partial^2_{pz}H -\frac{\epsilon\sigma^2}{2\delta}p/z)$.
\end{itemize}
Then, (RS-rMFG) has at most one smooth solution.
\end{thm}
\begin{proof}
See the Appendix.
\end{proof}

\begin{rem}
We observe that in contrast to Theorem  \ref{mainresultofPierreLouis} (risk-neutral case), the sufficiency condition for having at most one smooth solution in (RS-rMFG) now depends on the variance term.\ged
\end{rem}


\section{Numerical Illustration}\label{numerical}
In this section, we provide two numerical examples to illustrate the risk-sensitive mean-field game under  affine state dynamics and McKean-Vlasov dynamics.

\subsection{Affine state dynamics}


\begin{figure}
\begin{center}
  \includegraphics[scale=0.6]{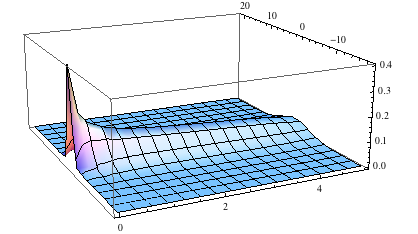}\\
  \caption{The evolution of distribution $m^*_t, 0\leq t\leq 5, -19\leq x\leq21$.}\label{mfg1M1018}
\end{center}
\end{figure}

\begin{figure}
\begin{center}
  \includegraphics[scale=0.6]{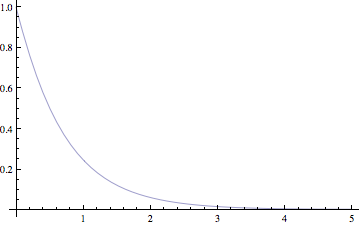}\\
  \caption{Mean value $\mathbb{E}(m^*_t)$ as a function of time, $0\leq t\leq 5$.}\label{mfg1Mean1018}
\end{center}
\end{figure}

\begin{figure}
\begin{center}
  \includegraphics[scale=0.6]{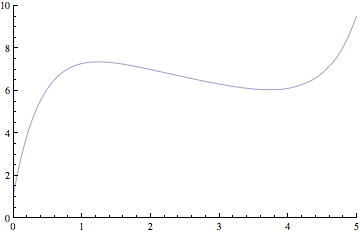}\\
  \caption{Variance of the distribution $m^*_t$ as a function of time, $0\leq t\leq 5$.}\label{mfg1Variance1018}
\end{center}
\end{figure}

\begin{figure}
\begin{center}
  \includegraphics[scale=0.6]{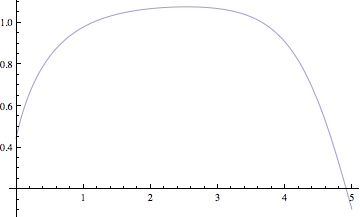}\\
  \caption{ $z(t)$ as a function of time, $0\leq t\leq T$.}\label{mfg1Z1018}
\end{center}
\end{figure}

We let Player $j$'s state evolution be described by a decoupled stochastic differential equation 
\begin{equation}
\nonumber dx_j^n(t)= u_j(t) dt+ \sqrt{\epsilon} \sigma d\mathbb{B}_j(t).
\end{equation}
The risk-sensitive cost functional is given by
\begin{equation}
\nonumber L(\bar{\gamma}_j, m^n; t, \underline{x}, \underline{m})=\delta \log \mathbb{E}_{\underline{x}, \underline{m}}\left\{\exp\left[\frac{1}{\delta}\left(Q(x^n_j)^2 \right.\right.\right.\left.\left.\left.+\int_0^T(q-\mathbb{E}(m^n_t))(x_j^n)^2(t)+\bar{u}_j^2(t)dt\right)\right]\right\},
\end{equation}
where $\delta, Q, q$ are positive parameters; hence coupling of the players is only through the cost. The optimal strategy of Player $j$ has the form of
\begin{equation}\label{agentOS}
\bar{u}_j^*(t)=-z(t)x,
\end{equation}
where $z(t)$ is a solution to the Riccati equation
\begin{equation}
\nonumber \dot{z}(t)+ q-\mathbb{E}(m^n)-z^2(t)(1-\sigma^2/\rho^2)=0,
\end{equation}
with boundary condition $z(T)=Q$.
An explicit solution is given by
\begin{equation}
\nonumber z(t)=-\frac{\sqrt{q-M}}{\sqrt{L}} \text{tan}\left[\sqrt{L} \sqrt{q-M} (t-T)+\right.\left.\text{arctan}\left(\frac{\sqrt{L} Q}{\sqrt{q-M}}\right)\right], 0\leq t \leq T,
\end{equation}
where $L:=1-\sigma^2/\rho^2$ and $M:=\mathbb{E}(m^n)$. The FPK-McV equation reduces to
\begin{equation}
\nonumber \partial_t m_t^* + \partial_x (m_t^* z(t)x(t)) =\frac{\epsilon}{2}\sigma^2\partial_{xx}^2 m^*_t.
\end{equation}
We set the parameters as follows: $q=1.2, Q=0.1$, $\delta=100,000$, $\sigma=2.0$, $T=5$ and $\epsilon=5.0$. Let $m^*_0(x)$ be a normal distribution $\mathcal{N}(1, 1)$ and for every $0\leq t\leq T$,  $m^*_t$ vanishes at infinity. In Figure \ref{mfg1M1018}, we show the evolution of the distribution $m^*_t$ and in Figures \ref{mfg1Mean1018} and \ref{mfg1Variance1018}, we show the mean and the variance of the distribution which affects the optimal strategies in (\ref{agentOS}). The optimal linear feedback $z(t)$ is illustrated in Figure \ref{mfg1Z1018}. We can observe that the mean value $\mathbb{E}(m_t^*)$ monotonically decreases from 1.0 and hence the unit cost on state is monotonically increasing. As the state cost increases, the control effort becomes relatively cheaper and therefore we can observe an increment in the magnitude of $z(t)$. However, when the mean value goes beyond 1.08, we observe that the control effort reduces to avoid undershooting in the state.

\subsection{McKean-Vlasov dynamics}
We let the dynamics of an individual player be
\begin{equation}\label{exdynamics}
dx_j^n(t) = \left( \frac{\beta}{n}\sum_{i=1}^n x_i^n(t) + u_j^n(t)  \right) dt + \sqrt{\epsilon}\sigma d\mathbb{B}_j(t),
\end{equation}
and take the risk-sensitive cost function to be
\begin{equation}
\nonumber  L = \delta \log E \left\{ \exp \left[ \frac{1}{\delta} \int_{0}^T q (x_j^n(t))^2 + (u_j^n(t))^2\right]\right\}.
\end{equation}

Note that the cost function is independent of other players' controls or states. As $n\rightarrow\infty$, under regularity conditions,
$$
\lim_{n\rightarrow\infty}\sum_{i=1}^n \frac{1}{n} x_i^n(t)= M(t),
$$
where $M(t)$ is the mean of the population. The feedback optimal control $\bar{u}_j$ in response to the mean field $M(t)$ is characterized by
\begin{equation}
\nonumber \bar{u}_j(t) = -z(t) x_j(t) - k(t),
\end{equation}
where
\begin{eqnarray}
\nonumber\dot{z}(t) + q - z^2 (1-\sigma^2/\rho^2) &=&0, \ \ z(T) = 0, \\
\nonumber\dot{k}(t) - z(t) k(t) + z(t) M(t) &=&0, \ \ k(T)=0, 
\end{eqnarray}
and $\rho^2=\frac{\delta}{2\epsilon}$ and $M(t)=\int_{x\in \mathcal{X}} xm(x,t) dx$. The Fokker-Planck-Kolmogorov (FPK) equation is
\begin{equation}
\nonumber\partial_t m(x, t) + \partial_x \left(m(x,t)   \left(-z(t) x(t) -k(t) \right)+ \beta m(x,t)\int_w wm(w, t) dw\right)  = \frac{\epsilon}{2}\sigma^2 \partial^2_{xx} m(x,t)
\end{equation}

By solving the ODEs, we find that
$$
z(t)=  -\sqrt{\bar{q}} \textrm{tan} \left( \sqrt{\hat{q}}(t-T)\right), \ \   0 \leq t \leq T.
$$
where $\bar{q} = q/ (1-\sigma^2/\rho^2)$ and $\hat{q} =  q (1-\sigma^2/\rho^2)$. Let $q=r=1$ and we find the solution
$$
k(t)= \textrm{cos}(t-T) \left( \int_1^T m(\tau)\textrm{sec}(T-\tau)\textrm{tan}(T-\tau) d\tau
-\int_t^T m(\tau')\textrm{sec}(T-\tau')\textrm{tan}(T-\tau') d\tau'
\right).
$$
Let $\sigma=1, \rho=2, \beta=1$ and we show in Figure \ref{IFACSimulationMV3fig1} the evolution of the probability density function $m(x, t)$. The mean $M(t)$ and the variance are shown in Figure \ref{IFACSimulationMV3fig2} and Figure \ref{IFACSimulationMV3fig3}, respectively.
\begin{figure}
\begin{center}
  \includegraphics[scale=0.8]{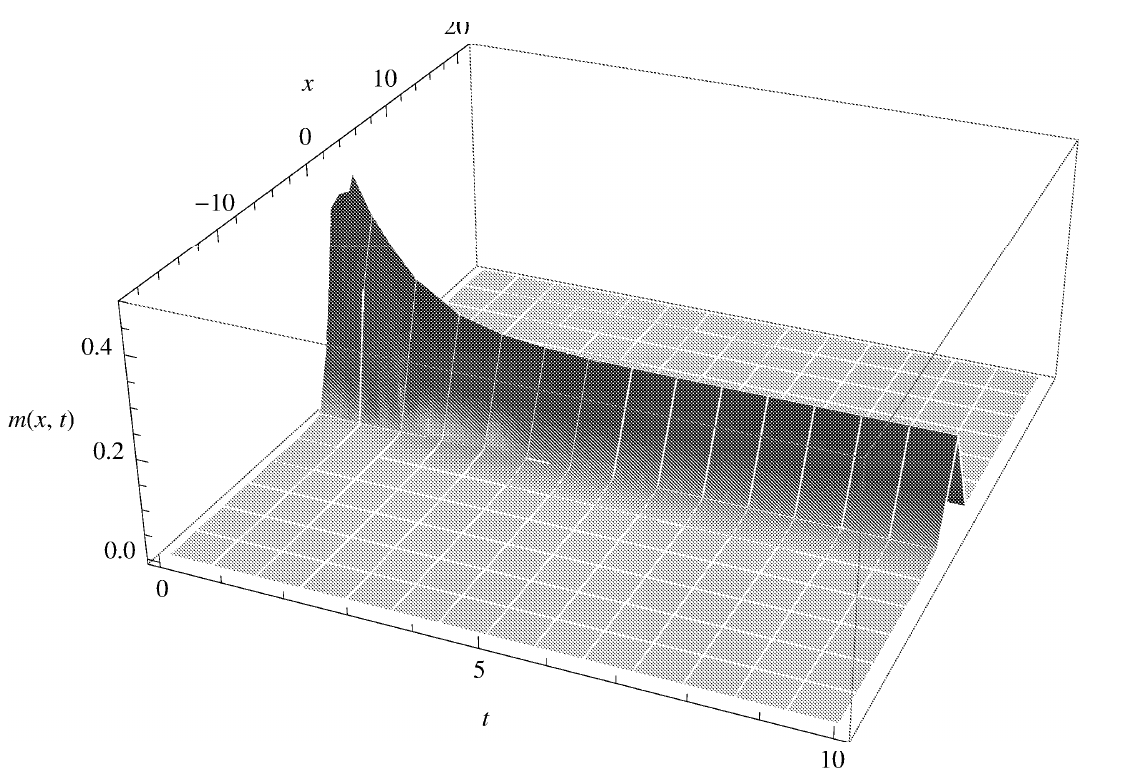}\\
  \caption{Evolution of the probability density function $m(x,t)$}\label{IFACSimulationMV3fig1}
\end{center}
\end{figure}

\begin{figure}
\begin{center}
  \includegraphics[scale=0.8]{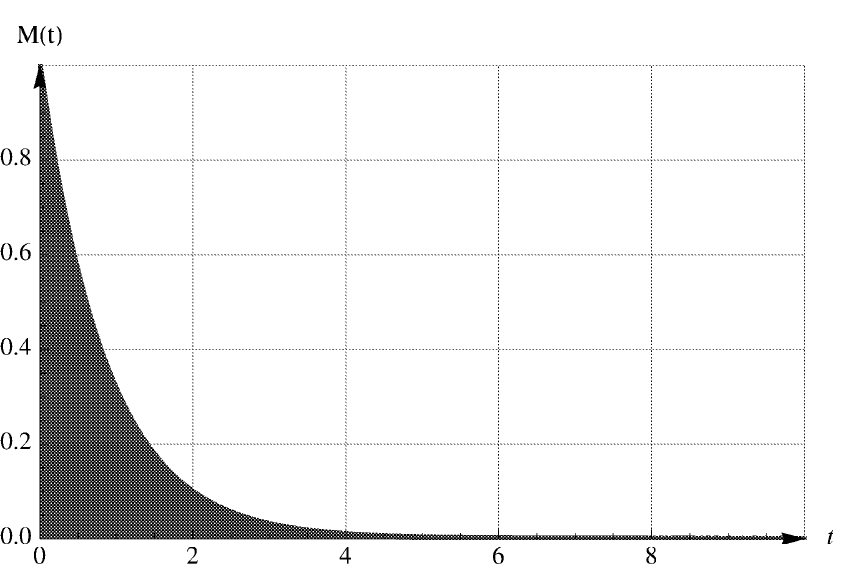}\\
  \caption{The mean $M(t)$ under equilibrium solution}\label{IFACSimulationMV3fig2}
\end{center}
\end{figure}

\begin{figure}
\begin{center}
  \includegraphics[scale=0.8]{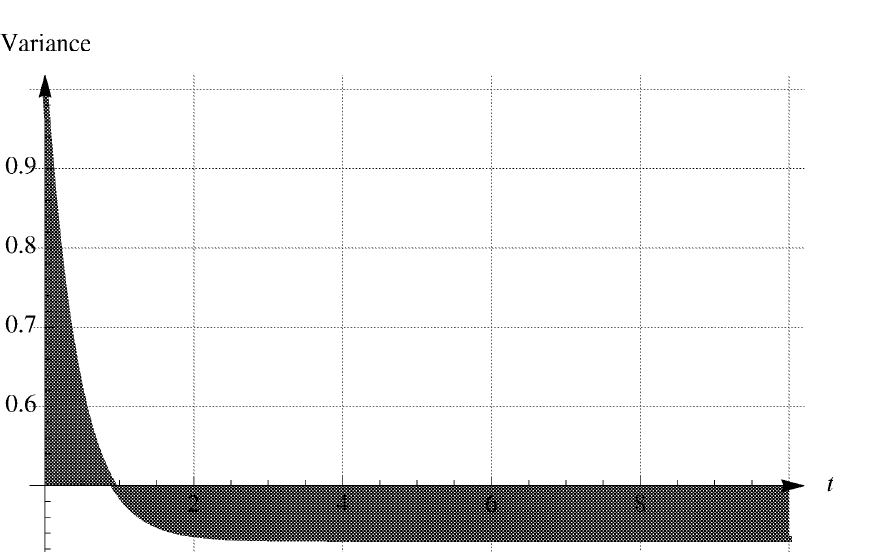}\\
  \caption{Variance over time under equilibrium solution}\label{IFACSimulationMV3fig3}
\end{center}
\end{figure}

\section{Concluding remarks} \label{tsec5}
We have studied risk-sensitive mean-field stochastic differential games with state dynamics  given by an It\^o stochastic differential equation and the cost function being the expected value of an exponentiated integral.

 Using a particular structure of state dynamics,
we have shown that the mean-field limit of the individual state dynamics leads to a controlled
macroscopic McKean-Vlasov equation. We have formulated a risk-sensitive mean-field response framework, and established its
compatibility with the density distribution using the controlled Fokker-Planck-Kolmogorov forward equation. The risk-sensitive mean-field equilibria are characterized by coupled backward-forward
equations. For the general case, the resulting mean field system is very hard to solve (numerically or analytically) even if the number of equations have been reduced. We have, however, provided generic explicit forms in the particular case of  the affine-exponentiated-Gaussian mean-field problem. In addition, we have shown that the risk-sensitive problem can be transformed into a risk-neutral mean-field game problem with the introduction of an additional fictitious player.  This allows one to study a novel class of mean field games, {\it robust mean field games}, under the Isaacs condition.

 An interesting direction that we leave for future research is to extend the model to accommodate  multiple classes of players and a drift function which may depend on the other players' controls. {\color{black} Another direction would be to soften the conditions under which Proposition 5 is valid, such as boundedness and Lipschitz continuity, and extend the result to games with non-smooth coefficients. In this context, one could address a mean field central limit question on the asymptotic behavior of the process $\sqrt{n}\mathbb{E}\left( \parallel x^n_j(t)-\tilde{x}_j(t)\parallel\right).$
 }
 Yet another extension would be to the time average risk-sensitive cost functional. Finally, the approach needs to be compared with other risk-sensitive approaches such as the {\it mean-variance criterion} and extended to the case where the drift is a function of the state-mean field and the control-mean field. 

\bibliographystyle{plain}

\bibliography{biblioMF1}

\appendix

%

{\color{black}
\begin{proof}[Proof of Proposition~\ref{prottt1}]
Under the stated standard assumptions on the drift $f$ and variance $\sigma$,
the forward stochastic differential equation has a unique solution adapted to the filtration generated by the Brownian motions.
We want to show that
$$\mathbb{E}\left( \sup_{t\in [0,T]}\parallel x^n_j(t)-\tilde{x}_j(t)\parallel\right)
\leq \frac{C_T}{\sqrt{n}},
$$
where $C_T$ is a positive number which only depends on the bounds, $T$ and the Lipschitz
constants of the coefficients of the drifts and the variance term.
First we observe that for a fixed control $u,$ the averaging terms $\frac{1}{n}\sum_{i=1}^n f(t,x_j,u,x_i)$ and $\frac{1}{n}\sum_{i=1}^n \sigma(t,x_j,u,x_i)$ are measurable, bounded and Lipschitz with the respect to the state  and uniformly with the respect to time.

Second, we observe that the bound on the Lipschitz constants of the coefficients do not depend on the population size $n.$

Hence, $\int f(t, x,u,x')\ m_t(dx')$  and $\int \sigma(t, x,u,x')\ m_t(dx')$ are bounded and Lipschitz uniformly with the respect to $t.$ Moreover, these coefficients are deterministic. This means that there is a unique solution to the limiting SDE and that solution is measurable with the filtration generated by the mutually independent Brownian motions.

Third,  we evaluate the gap between the coefficients in order to obtain an estimate of the two processes.
We start by evaluating the gap
$$
\mathbb{E}\left(\left\| \frac{1}{n}\sum_{i=1}^n f(t,x,u,x_i) -\int f(t,x,u,x')\ m_t(dx')\right\|^2\right)
$$

Notice that $f$ returns a $k-$dimensional vector and $x$ belongs to $\mathbb{R}^k$.
By reordering the above expression (in $2-$norm), we obtain
$$\sum_{l=1}^k \mbox{var}\left(\frac{1}{n}\sum_{i=1}^n f_{l}(t, x_j,u,x_i)\right)\leq \frac{k}{n} (1+\max_{l}b_l)^2\leq \frac{C_T}{n},$$ where $\mbox{var}(X)$ denotes the variance of $X$ and $b_l$ is a bound on the $l-$th component of the drift term. (This exists because we have assumed boundedness conditions on the coefficients).

Following a similar reasoning, we obtain the bounds on the second term in $\sigma$, i.e.,
$$\sum_{l,l'} \mbox{var}\left(\frac{1}{n}\sum_{i=1}^n \sigma_{ll'}(t, x_j,u,x_i)\right)\leq \frac{k}{n} (1+\max_{ll'}c_{ll'})^2\leq \frac{C_T}{n},$$
where $c_{ll'}$ is a bound on the entries $(l,l')-$ of the matrix $\sigma$.

Now we use the Lispchitz conditions and standard Gronwall estimates to deduce that
the mean of the quadratic gap between the two stochastic processes (starting from $x$ at time $0$)  is in order of $\frac{1}{n}.$

\end{proof}
}

{\color{black}
\begin{proof}[Proof of Theorem \ref{thmt1}]
Under the stated regularity and boundedness assumptions, there is a solution to the McKean-Vlasov FPK equation. Suppose that (i) and (ii) are satisfied.
 Then, $m_t=m^*(t,u^*(t))$ is the solution of the mean-field limit state dynamics, i.e., the macroscopic McKean-Vlasov PDE when $m$ is substituted into the HJB equation. By fixing  $f^*, c^*,\sigma,$ we obtain a novel HJB equation for the  mean-field stochastic game.
Since the new PDE admits a solution according to (ii), the control $u^*(t)=u(t,x)$ minimizing $\partial_xv \cdot f+c,$ is a best response to $m^*$  at time $t.$  The optimal response of the individual player generates a mean-field limit which in law is a solution of the FPK PDE and the players compute their controls as a function of this mean-field. Thus, the consistency between the control, the state and the mean field is guaranteed by assumption (i). It  follows that $(u^*,m^*)$ is a solution to the fixed-point problem i.e., a mean-field equilibrium,  and a strongly time-consistent one.

Now, we look at  the quadratic instantaneous cost case. In that case, we obtain the risk-sensitive  equations provided in Proposition 3. The fact that any convergent subsequence of best-response to $m^n$ is a best response to $m^*$ and the fact that $u^*$ is an $\epsilon^*-$best response to the mean-field limit $m^*$ follow from mean-field convergence of order $O\left(\frac{1}{\sqrt{n}}\right)$ and the continuity  of the risk-sensitive quadratic cost functional.
\end{proof}
}
\begin{proof}[Proof of Theorem \ref{secondmainresult}]

We  provide a sufficient condition for the risk-sensitive mean field game to have at most one smooth solution.
Suppose $\delta>0,$ and $\sigma$ is positive constant. Let $H$ be the Hamiltonian associated with the risk-neutral mean field system. Then the Hamiltonian for the risk-sensitive mean field system is $\tilde{H}(x,p,m)=H+(\frac{\epsilon\sigma^2}{2\delta})\parallel p\parallel^2.$ Assume that the dependence on $m$ is local, i.e., it is function of $m(x).$

 The generic expression for the optimal control  is given by
$u^*= \partial_p{H}(x,\partial_xv,m_t(x))$ (note that the generic feedback control is expressed in terms of $H$, and not of $\tilde{H}$).

Suppose that there exist two smooth solutions $(\hat{v}_1,\hat{m}_1),$ $(\hat{v}_2,\hat{m}_2)$ to the (normalized) risk-sensitive mean field system. Now, consider the function $t\longmapsto \int_{x\in\mathcal{X}} (\hat v_2(x)-\hat v_1(x))(\hat m_2(x)-\hat m_1(x))dx.$ Observe that this function is $0$ at time $t=0$ because the measures coincide initially, and the function is equal to $0$ at time $t=T$ because the final values coincide. Therefore, the function will be identically $0$ in $[0,T]$ if we show that it is monotone. This will imply that the integrand is zero, and hence one of the two terms $(\hat v_{2}(x)-\hat v_{1}(x))$ or $(\hat m_{2,t}(x)-\hat m_{1,t}(x))$ should be $0.$ Then, if the measures are identical, we use the HJB equation to obtain the result. If the value functions are identical, we can use the FPK equation to show the uniqueness of the measure.  Thus, it remains to find a sufficient condition for monotonicity, that is,
a sufficient condition under which the quantity $\int_{x\in\mathcal{X}} (\hat v_2(x)-\hat v_1(x))(\hat m_2(x)-\hat m_1(x))dx$ is monotone in time. We compute the following time derivative:
$$S(t):=\frac{d}{dt}\left[ \int_{x\in\mathcal{X}} (\hat v_2(x)-\hat v_1(x))(\hat m_2(x)-\hat m_1(x))dx \right].$$
We interchange the order of the integral  and the differentiation and use time derivative of a product to arrive at;
\begin{eqnarray}\nonumber
S(t)
&=& \int_{x\in\mathcal{X}} (\partial_t \hat v_2-\partial_t \hat v_1)(\hat m_2(x)- \hat m_1(x))dx + \\
\nonumber & & \int_{x\in\mathcal{X}}(\hat v_2-\hat v_1)(\partial_t \hat m_2(x)-\partial_t \hat m_1(x)) dx
\end{eqnarray}
Now we expand  the first term $ A:=\int_{x\in\mathcal{X}} (\partial_t \hat v_2-\partial_t \hat v_1)(\hat m_2(x)-\hat m_1(x))dx.$
Consider the two HJB equations:
\begin{eqnarray}
\nonumber \partial_t\hat v_1+\tilde{H}(x,\partial_x \hat v_1, \hat m_1(x))+\frac{1}{2}\sigma^2\partial^2_{xx}\hat v_1=0,\\
\nonumber \partial_t\hat v_2+\tilde{H}(x,\partial_x \hat v_2, \hat m_2(x))+\frac{1}{2}\sigma^2\partial^2_{xx}\hat v_2=0
\end{eqnarray}  To compute $A$, we take the difference between the two HJB equations above and multiply by $\hat m_2-\hat m_1,$ which gives
$$
\partial_t \hat v_2 -\partial_t\hat v_1=-\tilde{H}(x,\partial_x \hat v_2, \hat m_2)+ \tilde{H}(x,\partial_x \hat v_1, \hat m_1) -\frac{1}{2}\sigma^2  \partial^2_{xx}\hat v_2+
\frac{1}{2}\sigma^2 \partial^2_{xx}\hat v_1
$$
Hence,
\begin{eqnarray}
\nonumber{A}&:= & \int_x [\partial_t \hat v_2 -\partial_t\hat v_1](\hat m_2(x) -\hat m_1(x)) dx\\
\nonumber&=& -\int_x \tilde{H}(x,\partial_x \hat v_2, \hat m_2)(\hat m_2(x) -\hat m_1(x)) dx\\
\nonumber&& +\int_x \tilde{H}(x,\partial_x \hat v_1, \hat m_1) (\hat m_2(x) -\hat m_1(x)) dx\\
\nonumber&& -\int_x \frac{1}{2}\sigma^2\partial^2_{xx}(\hat v_2)(\hat m_2(x) -\hat m_1(x)) dx
\\ \nonumber && +
\int_x \sigma^2\frac{1}{2}\partial^2_{xx}(\hat v_1)(\hat m_2(x) -\hat m_1(x)) dx
\end{eqnarray}

Next we expand the second term
$B:=\int_{x\in\mathcal{X}} (\partial_t \hat m_2-\partial_t \hat m_1)(\hat v_2-\hat v_1)dx.$
%
%
%
Note that   the Laplacian terms are canceled by integration by parts in the expression $A+B$.
By collecting all  the terms in $A+B$, we obtain   \begin{eqnarray}
\nonumber A+B &=&  -\int_x \tilde{H}(x,\partial_x \hat v_2, \hat m_2)(\hat m_2(x) -\hat m_1(x)) dx\\
\nonumber&& +\int_x \tilde{H}(x,\partial_x \hat v_1, \hat m_1) (\hat m_2(x) -\hat m_1(x)) dx\\
\nonumber&&+\int_x \hat m_2(x) [\partial_p{H}(x,\partial_x\hat v_2, \hat m_2)](\partial_x\hat v_2-\partial_x\hat v_1)dx\\
\nonumber &&
 -\int_x\hat m_1(x) [\partial_p{H}(x,\partial_x \hat v_1, \hat m_1)](\partial_x\hat v_2-\partial_x\hat v_1)dx
\end{eqnarray}
Letting $S(t)=A+B$, we introduce
$$
\hat m_{\lambda}:=(1-\lambda)\hat m_{1}+\lambda \hat m_{2}=\hat m_{1}+\lambda (\hat m_{2}-\hat m_{1}).
$$ The measure $\hat m_{\lambda}$ starts with $\hat m_{1}$ for the parameter $\lambda=0$ and yields the measure $ \hat m_2$ for $\lambda=1.$ Similarly define
$$
\hat v_{\lambda}:=(1-\lambda)\hat v_{1}+\lambda \hat v_{2}.
$$

Introduce
an auxiliary integral parameterized by $\lambda.$
\begin{eqnarray}
 \nonumber  C_{\lambda}
 &:=&  -\int_x \tilde{H}(x,\partial_x \hat v_{\lambda}, \hat m_{\lambda})(\hat m_{\lambda}(x) -\hat m_1(x)) dx\\
\nonumber && +\int_x \tilde{H}(x,\partial_x \hat v_1, \hat m_1) (\hat m_{\lambda}(x) -\hat m_1(x)) dx\\
\nonumber &&+\int_x \hat m_{\lambda}(x) [\partial_p{H}(x,\partial_x\hat v_{\lambda}, \hat m_{\lambda})](\partial_x\hat v_{\lambda}-\partial_x\hat v_1)dx\\
  \nonumber &&
 -\int_x\hat m_1(x) [\partial_p{H}(x,\partial_x \hat v_1, \hat m_1)](\partial_x\hat v_{\lambda}-\partial_x\hat v_1)dx
\end{eqnarray} Substituting the terms $\hat v_{\lambda}-\hat v_1=\lambda (\hat v_{2}-\hat v_1)$ and
$\hat m_{\lambda}-\hat m_1=\lambda (\hat m_{2}-\hat m_1),$ we obtain
\begin{eqnarray}
\nonumber \frac{C_{\lambda}}{\lambda}
 &:=&  -\int_x \tilde{H}(x,\partial_x \hat v_{\lambda}, \hat m_{\lambda})(\hat m_{2}(x) -\hat m_1(x)) dx\\
\nonumber && +\int_x \tilde{H}(x,\partial_x \hat v_1, \hat m_1) (\hat m_{2}(x) -\hat m_1(x)) dx\\
\nonumber&&+\int_x \hat m_{\lambda}(x) [\partial_p{H}(x,\partial_x\hat v_{\lambda}, \hat m_{\lambda})](\partial_x\hat v_{2}-\partial_x\hat v_1)dx\\
\nonumber &&
 -\int_x\hat m_1(x) [\partial_p{H}(x,\partial_x \hat v_1, \hat m_1)](\partial_x\hat v_{2}-\partial_x\hat v_1)dx
\end{eqnarray}

Using the continuity of the terms (of the RHS) above and the compactness of $\mathcal{X},$ we deduce that
$$ \lim_{\lambda\longrightarrow 0}\frac{C_{\lambda}}{\lambda}=0.
$$
{\color{black}
 We next find a condition under which  the one-dimensional function $ \lambda \longmapsto \frac{C_{\lambda}}{\lambda}$ is monotone in $\lambda.$
 We need to compute the variations of
 $$\frac{d}{d\lambda}\left( \frac{C_{\lambda}}{\lambda}\right).$$ Suppose that
$(x,p,m)\longmapsto \tilde{H}(x,p,m)$ is twice continuously differentiable with the respect to $(p,m).$ Then,
 \begin{eqnarray}
\nonumber  \frac{d}{d\lambda}\left( \frac{C_{\lambda}}{\lambda}\right) &= & -\int_x
  \left[\partial_p\tilde{H}(x,\partial_x \hat v_{\lambda}, \hat m_{\lambda})(\partial_x\hat v_{2}-\partial_x\hat v_1)\right](\hat m_{2}(x) -\hat m_1(x)) dx\\ \nonumber && -\int_x
  \left[ \partial_m\tilde{H}(x,\partial_x \hat v_{\lambda}, \hat m_{\lambda})(\hat m_{2}(x) -\hat m_1(x)) \right](\hat m_{2}(x) -\hat m_1(x)) dx\\
  \nonumber &&
  +\int_x \partial_{\lambda}\left( \hat m_{\lambda}(x)[\partial_p{H}(x,\partial_x\hat v_{\lambda}, \hat m_{\lambda})]\right)(\partial_x\hat v_{2}-\partial_x\hat v_1)dx
 \end{eqnarray}
 \begin{eqnarray}
 \nonumber \frac{d}{d\lambda}\left( \frac{C_{\lambda}}{\lambda}\right) &= & -\int_x
  \partial_p\tilde{H}(x,\partial_x \hat v_{\lambda}, \hat m_{\lambda})(\partial_x\hat v_{2}-\partial_x\hat v_1)(\hat m_{2}(x) -\hat m_1(x)) dx\\
  \nonumber&&
  -\int_x
  \partial_m\tilde{H}(x,\partial_x \hat v_{\lambda}, \hat m_{\lambda})(\hat m_{2}(x) -\hat m_1(x))^2 dx \\ \nonumber
  &&
  +\int_x(\hat m_{2} -\hat m_1)
  [\partial_p{H}(x,\partial_x\hat v_{\lambda}, \hat m_{\lambda})]
  (\partial_x\hat v_{2}-\partial_x\hat v_1)dx
  \\ \nonumber &&
  +\int_x \hat m_{\lambda}
  \partial_{\lambda}
  [\partial_p{H}(x,\partial_x\hat v_{\lambda}, \hat m_{\lambda})]
  (\partial_x\hat v_{2}-\partial_x\hat v_1)dx
 \end{eqnarray}

{Computation of the term $ \hat m_{\lambda}(x)\partial_{\lambda}\left(  [\partial_p{H}(x,\partial_x\hat v_{\lambda}, \hat m_{\lambda})]\right)$} yields
 \begin{eqnarray} \nonumber
 D_{\lambda} &=&
\partial_{\lambda} [\partial_p{H}(x,\partial_x\hat v_{\lambda}, \hat m_{\lambda})]\\
\nonumber &=&\partial^2_{pp}{H}. (\partial_x\hat v_{2}-\partial_x\hat v_1)+\partial^2_{mp} {H}. (\hat m_{2} -\hat m_1)
 \end{eqnarray}
and we obtain
\begin{eqnarray}
 \nonumber \frac{d}{d\lambda}\left( \frac{C_{\lambda}}{\lambda}\right) &= & -\int_x
  \partial_p\tilde{H}(x,\partial_x \hat v_{\lambda}, \hat m_{\lambda})(\partial_x\hat v_{2}-\partial_x\hat v_1)(\hat m_{2}(x) -\hat m_1(x)) dx\\
  \nonumber &&
  -\int_x
  \partial_m\tilde{H}(x,\partial_x \hat v_{\lambda}, \hat m_{\lambda})(\hat m_{2}(x) -\hat m_1(x))^2 dx \\
\nonumber  &&
  +\int_x(\hat m_{2} -\hat m_1)
  [\partial_p{H}(x,\partial_x\hat v_{\lambda}, \hat m_{\lambda})]
  (\partial_x\hat v_{2}-\partial_x\hat v_1)dx
  \\ \nonumber &&
  +\int_x \hat m_{\lambda} \partial^2_{pp}{H}. (\partial_x\hat v_{2}-\partial_x\hat v_1)^2+\hat m_{\lambda}\partial^2_{mp} {H}. (\hat m_{2} -\hat m_1)(\partial_x\hat v_{2}-\partial_x\hat v_1)
 \end{eqnarray}

The first and the third lines differ by $$ -\int_x\left(\frac{\epsilon\sigma^2}{\delta}\langle .,\nabla_x \hat v \rangle\right) \left(\partial_x\hat v_{2}-\partial_x\hat v_1\right)\left(\hat m_{2}(x) -\hat m_1(x)\right) dx.
$$

 Hence, we obtain
\begin{eqnarray}
\nonumber  \frac{d}{d\lambda}\left( \frac{C_{\lambda}}{\lambda}\right) &= &
  \int_x m_{\lambda}( \partial_x\hat v_{2}-\partial_x\hat v_1, \hat m_{2} -\hat m_1)\left(\begin{array}{cc}
 a_{11}& a_{12} \\
 a_{21}&  a_{22}
\end{array} \right) \left(\begin{array}{c}
 \partial_x\hat v_{2}-\partial_x\hat v_1\\ \hat m_{2} -\hat m_1
\end{array} \right)dx,
  \end{eqnarray}
  where
  $$
  a_{11}: = \partial_{pp}^2 {H},
  $$

  $$a_{21} :=\frac{1}{2}\partial_{mp}^2 \tilde H=\frac{1}{2}\partial_{mp}^2  H-\frac{\epsilon\sigma^2}{2\delta}p/m,
  $$

  $$a_{21}:= \frac{1}{2}(\partial_{pm}^2 \tilde H)'-\frac{\epsilon\sigma^2}{2\delta}\frac{p}{m}=\frac{1}{2}(\partial_{pm}^2  H)'-\frac{\epsilon\sigma^2}{2\delta}\frac{p}{m},$$

  $$a_{22}:=- \frac{\partial_m \tilde{H}}{m}.$$

Suppose that for all $(x,p,m)\in \mathcal{X}\times \mathbb{R}^d\times \mathbb{R}_{+},$ the matrix
$$\left( \begin{array}{cc}
a_{11} & a_{12}\\
a_{21} & a_{22}
\end{array}\right)\succ 0.
$$
Then, the monotonicity follows, and this completes the proof.
}
\end{proof}
\end{document}